\documentclass[a4paper,10pt]{amsart}
\usepackage[english]{babel}
\usepackage[T1]{fontenc}
\usepackage[utf8]{inputenc}
\usepackage{xpatch}
\usepackage{stix2}
\usepackage{csquotes}
\usepackage{geometry}
\usepackage{amsmath,amssymb,amsthm}
\usepackage{IEEEtrantools}
\usepackage{mathtools}
\usepackage{xcolor}
\usepackage[titletoc,title]{appendix}
\usepackage{dsfont}
\usepackage{microtype}
\usepackage[font=footnotesize, margin=1em]{caption}
\usepackage{hyperref}
\usepackage{cleveref}

\definecolor{sienne}{RGB}{136, 45, 23}
\hypersetup{
colorlinks = true,
linkcolor = sienne,
citecolor = blue
}
\allowdisplaybreaks[1]

\newtheorem{theorem}{Theorem}
\newtheorem{prop}[theorem]{Proposition}

\theoremstyle{definition}
\newtheorem{definition}[theorem]{Definition}
\theoremstyle{remark}
\newtheorem{remark}[theorem]{Remark}
\newtheorem{ex}[theorem]{Example}

\mathtoolsset{mathic}

\makeatletter
\patchcmd{\@IEEEeqnarray}{\relax}{\relax\intertext@}{}{}
\makeatother

\DeclareMathOperator{\distance}{distance}

\DeclareMathOperator{\Supp}{Supp}

\DeclareMathOperator{\Leb}{Leb}

\let\Re\relax
\let\Im\relax
\DeclareMathOperator{\Re}{Re}
\DeclareMathOperator{\Im}{Im}
\DeclareMathOperator{\Reelle}{Re}

\newcommand{\smallO}{o}
\newcommand{\bigO}{O}

\newcommand{\diff}[1][-3]{\ensuremath{\mathop{}\mkern#1mu\mathrm{d}}}

\newcommand{\set}{\mathbb}
\newcommand{\N}{{\set{N}}}
\newcommand{\Z}{{\set{Z}}}
\newcommand{\R}{{\set{R}}}
\newcommand{\C}{{\set{C}}}
\newcommand{\T}{{\set{T}}}

\newcommand{\Cc}{{\mathcal C}}

\newcommand{\coloneqq}{\mathrel{\mathord{:}\mathord=}}
\newcommand{\eq}{\Leftrightarrow}

\newcommand{\per}[1]{#1\mkern0mu_{\mathrm{per}}}

\newcommand{\eu}{{\mathrm e}}
\newcommand{\iu}{{\mathrm i}}

\makeatletter
\newcounter{proofstep}[theorem]
\newcommand{\step}[1]{%
\refstepcounter{proofstep}%
\vskip-\lastskip\medskip\noindent\textit{Step \arabic{proofstep}: #1. --- }}
\xapptocmd\proof{\setcounter{proofstep}{0}}{}{}
\makeatother

\title{Null-controllability for weakly dissipative heat-like equations}

\author{Paul Alphonse}
\address{(Paul Alphonse) Universit\'e de Lyon, ENSL, UMPA - UMR 5669, F-69364 Lyon}
\email{paul.alphonse@ens-lyon.fr}

\author{Armand Koenig}
\thanks{This work has been partially supported by the ANR LabEx CIMI (under grant ANR-11-LABX-0040) within the French State Programme “Investissements d’Avenir”.}
\address{(Armand Koenig) IMT, Université de Toulouse, CNRS, Université Toulouse III - Paul Sabatier, Toulouse, France}
\email{armand.koenig@math.univ-toulouse.fr}


\makeatletter
	\@namedef{subjclassname@2020}{\textup{2020} Mathematics Subject Classification}
\makeatother


\begin{document}
\maketitle

\begin{abstract} We study the null-controllability properties of heat-like equations posed on the whole Euclidean space $\mathbb R^n$. These evolution equations are associated with Fourier multipliers of the form $\rho(\vert D_x\vert)$, where $\rho\colon[0,+\infty)\rightarrow\mathbb C$ is a measurable function such that $\Re\rho$ is bounded from below. 
We consider the ``weakly dissipative'' case, a typical example of which is given by the fractional heat equations associated with the multipliers $\rho(\xi) = \xi^s$ in the regime $s\in(0,1)$, for which very few results exist. We identify sufficient conditions and necessary conditions on the control supports for the null-controllability to hold. 
More precisely, we prove that these equations are null-controllable in any positive time from control supports which are sufficiently thick at all scales. Under assumptions on the multiplier $\rho$, in particular assuming that $\rho(\xi) = o(\xi)$, we also prove that the null-controllability implies that the control support is thick at all scales, with an explicit lower bound of the thickness ratio in terms of the multiplier $\rho$.
Finally, using Smith-Volterra-Cantor sets, we provide examples of non-trivial control supports that satisfy these necessary or sufficient conditions.
\end{abstract}

\section{Introduction}

\subsection{Motivation} We study the null-controllability properties of the following class of parabolic heat-like equations
\begin{equation}\label{eq-gen}\tag{$E_\rho$}
\begin{cases}
	\partial_tf(t,x) + \rho(\vert D_x\vert)f(t,x) = \mathds 1_\omega u(t,x), \quad (t,x) \in \R^*_+ \times \R^n,\\
	f(0,\cdot) = f_0\in L^2(\R^n).
\end{cases}
\end{equation}
Above, the operator $\rho(\vert D_x\vert)$ is the Fourier multiplier associated with the symbol $\rho(\vert\xi\vert)$, with $\vert\cdot\vert$ the canonical Euclidean norm in $\mathbb R^n$, the function $\rho\colon[0,+\infty)\rightarrow\mathbb C$ being measurable such that $\Re\rho$ is bounded from below, and $\omega\subset\mathbb R^n$ is a measurable set with positive Lebesgue measure. We investigate the relationship between the geometry of $\omega$ and the null-controllability properties of these heat-like equations, defined as follows.

\begin{definition} Let $T>0$ and $\omega\subset\mathbb R^n$ be a measurable set with positive measure. The equation \eqref{eq-gen} is said to be \textit{null-controllable} from $\omega$ in time $T>0$ when for all $f_0\in L^2(\mathbb R^n)$, there exists a control $u\in L^2((0,T)\times\omega)$ such that the mild solution of \eqref{eq-gen} satisfies $f(T,\cdot) = 0$.
\end{definition}

Although the null-controllability properties of parabolic equations posed on bounded domains of $\mathbb R^n$ have been known for years \cite{LR,MZ,M2006,AE,AEWZ,BM}, the same study for parabolic equations posed on the whole Euclidean space $\mathbb R^n$, as the equations \eqref{eq-gen}, is quite recent. It follows from previous works \cite{AB,AM,BEPS,ES,EV,HWW,LWXY,TWX,WWZZ} that the null-controllability properties of such models, and also their approximate null-controllability or the stabilization properties, are associated with the geometric notion of thickness, defined as follows.
\begin{definition}
Given $\gamma\in(0,1)$ and $r>0$, the set $\omega\subset\mathbb R^n$ is said to be $\gamma$-\textit{thick} at scale $r$, or $(\gamma,r)$-thick, when it is measurable and satisfies
\[\forall x\in\mathbb R^n,\quad\Leb(\omega\cap B(x,r))\geq\gamma\Leb(B(x,r)),\]
where $\Leb$ denotes the Lebesgue measure in $\mathbb R^n$.
\end{definition}

Precisely, it is known from the work \cite{AM} that the thickness is a geometric necessary condition for the null-controllability of the general equations \eqref{eq-gen} (in fact, more generally, for the rapid stabilization of such equations). This condition also turns out to be a necessary and sufficient condition that ensures the null-controllability in any positive time $T>0$ of the fractional heat equations $($\hyperref[eq-gen]{$E_{\rho_{s}}$}$)$ associated with the multipliers $\rho_s(t) = t^s$ in the regime $s>1$ \cite{AB,EV,NTTV, WWZZ}. For this particular class of equations, which will serve as a common thread throughout this introduction, it was also proven in \cite{Ko} that in the weak-dissipation regime $s\in(0,1)$, positive null-controllability results can not be obtained from control supports $\omega\subset\mathbb R^n$ which are not dense in the whole space $\mathbb R^n$, this result being also established in the critical dissipation regime $s=1$ (corresponding to the half-heat equation) in dimension $n=1$~\cite[Théorème 2.3]{K}\cite{L}. 

From this point on, two areas of work naturally emerge. On the one hand, it would be interesting to characterize the multipliers $\rho$ for which the thickness is a necessary and sufficient geometric condition that ensures the null-controllability of the associated evolution equations \eqref{eq-gen}, as with the strongly-dissipative heat equations (it is worth noting that the generalized Lebeau-Robianno's method as stated by Duyckaerts and Miller~\cite[Theorem 6.1]{DM} together with Kovrijkine's spectral estimate~\cite{Kovrijkine} gives a first result on this topic). One the other hand, it would be interesting to study the null-controllability properties of the heat-like equations \eqref{eq-gen} in a weak-dissipation setting, for which very few results have been obtained so far, and to continue the studies carried out in the works \cite{Ko,K,L}, in particular by looking for control supports from which these equations can be controlled to zero. In the following, we will only focus on this second point.

\subsection{Main results}\label{sec-results}
In the present work, we prove that the null-controllability properties of the parabolic heat-like equations \eqref{eq-gen} is associated with the following stronger notion of thickness in the weak-dissipation regime.

\begin{definition}\label{def-a-thick}
Given some $r_0>0$ and a function $\gamma\colon(0,r_0]\rightarrow[0,1]$, a measurable set $\omega\subset\mathbb R^n$ is said to be \textit{thick with respect to} $\gamma$ (or $\gamma$-\textit{thick}) when it satisfies that for every $r\in(0,r_0]$ and $x\in \R^n$,
\[\Leb(\omega\cap B(x,r))\geq\gamma(r)\Leb(B(x,r)).\]
\end{definition}

This definition can be rephrased as ``$\omega$ is $\gamma(r)$-thick at every scale $r\in(0,r_0]$''. Therefore, being thick with respect to $\gamma$ is a far stronger notion than the usual notion of thickness.

We first prove a general result stating that the parabolic equation \eqref{eq-gen} is always null-controllable from control supports $\omega\subset\mathbb R^n$ being thick with respect to some function $\gamma_{\rho}$ associated to the multiplier $\rho$ (under reasonnable assumptions).

\begin{theorem}\label{thm-main2} Let $\rho\colon[0,+\infty)\rightarrow\mathbb C$ be a function such that $\Re\rho$ is a non-negative continuous function satisfying $\lim_{+\infty}\Reelle\rho=+\infty$. Let $r_0>0$ and $\gamma_{\rho}\colon(0,r_0]\rightarrow(0,1]$ be the function defined by 
\begin{equation}\label{eq-gammarho}
    \gamma_{\rho}(r)\coloneq c_0\exp(-c_1(\Reelle\rho)(1/r)^{\alpha}),
\end{equation}
where $c_0\in(0,1)$, $c_1>0$ and $\alpha\in(0,1)$ are some parameters. Let $\omega \subset \set R^n$ be $\gamma_\rho$-thick. Then, for every $T>0$, the parabolic equation \eqref{eq-gen} is null-controllable from $\omega$ in time $T$.
\end{theorem}

Under additional assumptions on the multiplier $\rho$, we also prove that the thickness with respect to some function $\tilde \gamma_\rho\colon (0,r_0]\rightarrow[0,1]$ is also a necessary condition for the null-controllability to hold.

\begin{theorem}\label{thm-mainIII}
Let $K>0$ and $\Cc = \{\xi \in \set C, \Re(\xi)>K, 
\lvert\Im(\xi)|<K^{-1}\Re(\xi)\}$. Let $\rho\colon \Cc\cup \R_+ \to \C$ be such that
\begin{itemize}
    \item $\rho$ is holomorphic on $\Cc$,
    \item $\rho(\xi) = \smallO(\xi)$ as $|\xi|\to \infty$, $\xi \in \Cc$,
    \item $\rho$ is measurable on $\R_+$ and $\inf_{\R_+} \Re(\rho) >-\infty$,
    \item there exists $C>0$ such that for $\xi\in \Cc$, $\lvert\Im(\rho(\xi))|\leq C\Re(\rho(\xi))$,
    \item $\ln(\xi) = \smallO (\Re\rho(\xi))$ in the limit $|\xi|\to +\infty$, $\xi\in \Cc$.
\end{itemize}

Let $T>0$ and $\omega\subset \R^n$ be measurable. Assume that the parabolic equation~\eqref{eq-gen} is null controllable from $\omega$ in time $T>0$. 

There exists $\lambda>0$ and $c>0$ such that if $\epsilon>0$, and if $r\mapsto h_r$ is defined for small enough $r$ and such that for every $r$ small enough,
\begin{equation}\label{eq-tech}
    \sqrt{h_r (2T+\epsilon) \Re\rho\left(\frac{\lambda}{h_r}\right)} \leq r,
\end{equation}
then, for every small enough $r> 0$, and $x\in \R^n$, we have
\[
\frac{\Leb(\omega\cap B(x,r))}{\Leb(B(x,r))} 
\geq cr^{-n}\exp\left(-2(T+\epsilon) \Re\rho\left(\frac{\lambda}{h_r}\right)\right).
\]
\end{theorem}

\begin{remark}\hspace{0pt}
\begin{itemize}
    \item The hypothesis $\rho(\xi) = \smallO(\xi)$ is a rigorous way of saying that the heat-like equation~\eqref{eq-gen} is weakly dissipative. Since the null-controllability is known to hold for strongly dissipative equations (e.g., the heat equation) on any thick set, one cannot expect to obtain \cref{thm-mainIII} without a hypothesis of this kind.
    \item The hypotheses $\lvert\Im(\rho(\xi))|\leq C\Re(\rho(\xi))$ and $\ln(\xi) = \smallO (\Re\rho(\xi))$ are mainly assumed for cosmetic reasons. We could prove some results with weaker hypotheses, but the result (and the proof) would be even more tedious. This would be of dubious interest, as such, we prefer not to detail these results here.
    \item The holomorphy hypothesis is a technical limitation of our strategy of proof involving complex deformation of integration path. Proving a version of \cref{thm-mainIII} without this hypothesis is an open problem.
\end{itemize}
\end{remark}

\begin{remark}
Given some function $\gamma\colon(0,r_0]\rightarrow[0,1]$, an example of set being $\gamma$-thick is of course the whole space $\mathbb R^n$, but it might be difficult to visualize non trivial examples of sets satisfying this property. In \cref{sec-examples}, we construct subsets $\omega$ of $\mathbb R^n$ which are $\gamma$-thick and such that $\Leb(\set R^n \setminus \omega)>0$. Roughly speaking, when $\gamma$ is assumed to be decreasing and satisfying $\gamma(r) \to 0$ as $r \to 0$, these are complements of Smith-Volterra-Cantor sets associated with the sequence $\big(24(\gamma(2^{-n}) - \gamma(2^{-n-1}))\big)_{n\geq0}$ (whose definition we recall in \cref{def-SVC}). We refer to \cref{th-omega-gamma-thick} for the details.
\end{remark}

Let us now apply \cref{thm-main2} and \cref{thm-mainIII} to the fractional heat equations.

\begin{ex}\label{ex-frac} For all positive real number $s>0$, let us consider the multiplier $\rho_s\colon[0,+\infty)\rightarrow[0,+\infty)$ defined for all $t\geq0$ by $\rho_s(t) = t^s$. Let us consider a positive time $T>0$ and a measurable set $\omega\subset\mathbb R^n$. As recalled in the beginning of the introduction, the null-controllability properties of the associated fractional heat equation $($\hyperref[eq-gen]{$E_{\rho_{s}}$}$)$ are well understood in the high-dissipation regime $s>1$. We will therefore only focus on the weak-dissipation regime $s\in(0,1]$. 

On the one hand, it follows from \cref{thm-main2} that in the regime $s\in(0,1]$, and when the set $\omega$ is thick with respect to the function $\gamma_s(r) = c_0\exp(-c_1r^{\alpha s})$, where $c_0,c_1>0$ and $\alpha\in(0,1)$ are parameters, then the fractional heat equation $($\hyperref[eq-gen]{$E_{\rho_{s}}$}$)$ is null-controllable from $\omega$ at time $T$. As far as we know, this is the first positive null-controllability result for the fractional heat equation in the weak-dissipative regime $s\in(0,1]$.

On the other hand, notice that when $s\in(0,1)$, the condition \eqref{eq-tech} of \cref{thm-mainIII} is satisfied with $h_r = r^{2/(1-s)}$ for the above multiplier $\rho_s$. As a consequence, still in the regime $s\in(0,1)$, it follows from \cref{thm-mainIII} that if the fractional heat equation $($\hyperref[eq-gen]{$E_{\rho_{s}}$}$)$ is null-controllable from $\omega$ at time $T$, then there exist some positive constants $c_0,c_1>0$ such that $\omega$ is thick with respect to the function $\gamma_s(r) = c_0\exp(-c_1r^{2s/(1-s)})$. Notice that we do not consider the critical case $s=1$, whose understanding remains an open problem.
\end{ex}

\begin{remark}
Let us consider the fractional heat equation associated to the Fourier multiplier $\rho_s(t) = t^s$ as above, with $s\leq1$. In dimension one, one popular way to study the null-controllability of PDEs is the \emph{moment method}. The is the strategy employed by Micu and Zuazua~\cite{MZ}, using \emph{shaped controls}, i.e., controls of the form $u(t)h(x)$.

Our results underline the difference between \emph{shaped controls} and \emph{internal controls} (the kind of controls we are considering). Indeed, \cref{ex-frac} show that if $\omega$ is sufficiently thick, the fractional heat equation is null controllable with internal controls; but if we consider \emph{shaped controls}, then null-controllability \emph{never} holds, whatever the profile $h$ is~\cite[Section~5]{MZ} (see also \cite[Appendix]{M2006}).
\end{remark}

\section{Sufficient condition}\label{sec-sufficient}

This section is devoted to the proof of \cref{thm-main2}, which states that given a function $\rho:[0,+\infty)\rightarrow\mathbb C$ such that $\Reelle\rho$ is a continuous non-negative function such that $\lim_{+\infty}\Reelle\rho=+\infty$, the parabolic equation \eqref{eq-gen} is null-controllable from any set $\omega\subset\mathbb R^n$ being thick with respect to the density $\gamma_{\rho}\colon(0,r_0]\rightarrow(0,1]$ defined by eq.~\eqref{eq-gammarho}, and in any positive time $T>0$.

By the Hilbert Uniqueness Method (see, e.g., \cite[Theorem 2.44]{coron_book}), the null-controllability of the parabolic equation \eqref{eq-gen} is equivalent to the observability of the heat-like semigroup $(\eu^{-t\overline\rho(\vert D_x\vert)})_{t\geq0}$, that we recall in the following definition.

\begin{definition} 
Let $T>0$, and let $\omega \subset \R^n$ be measurable. The semigroup $(\eu^{-t\overline\rho(\vert D_x\vert)})_{t\geq0}$ is said to be observable from the set $\omega$ in time $T$ if there exists a positive constant $C_{\omega,T} > 0$ such that for all $g\in L^2(\R^n)$,
\[\Vert \eu^{-T\overline\rho(\vert D_x\vert)}g\Vert^2_{L^2(\R^n)}\le C_{\omega,T} \int_0^T\Vert \eu^{-t\overline\rho(\vert D_x\vert)}g\Vert^2_{L^2(\omega)}\,\diff t.\]
\end{definition}

We will prove that the heat-like semigroup $(\eu^{-t\overline\rho(\vert D_x\vert)})_{t\geq0}$ is indeed observable, with an upper bound on the observability constant. This will imply that the heat-like equation~\eqref{eq-gen} is null-controllable.

\begin{theorem}\label{thm-obsgen}
Let $c_0\in(0,1)$, $c_1>0$, $\alpha \in (0,1)$ and let $\gamma_\rho$ be defined by \cref{eq-gammarho}. Let $\omega\subset\mathbb R^n$ be $\gamma_{\rho}$-thick. There exists a positive constant $C>0$ such that for all $T>0$ and $g\in L^2(\mathbb R^n)$,
\[
    \Vert \eu^{-T\overline\rho(\vert D_x\vert)}g\Vert^2_{L^2(\R^n)}\le C\exp\bigg(\frac C{T^{\alpha/(1-\alpha)}}\bigg)\int_0^T\Vert \eu^{-t\overline\rho(\vert D_x\vert)}g\Vert^2_{L^2(\omega)}\,\diff t,
\]
where $\alpha\in(0,1)$ is the parameter appearing in the definition of the function $\gamma_{\rho}$.
\end{theorem}

The rest of this section is devoted to the proof of \cref{thm-obsgen}. This will be done in two steps: first proving a spectral estimate reminiscent of Jerison and Lebeau's spectral inequality~\cite[Theorem~14.6]{JL96} or Logvinenko-Sereda-Kovrijkine estimate~\cite{Kovrijkine}, and second using Lebeau and Robbiano's method, as stated in the following theorem proven by Beauchard and Pravda-Starov.

\begin{theorem}[Theorem 2.1 in \cite{BPS}]\label{thm-obs}
Let $\omega$ be a measurable subset of $\mathbb{R}^n$ with positive Lebesgue measure, $(\pi_k)_{k\geq1}$ be a family of orthogonal projections defined on $L^2(\mathbb{R}^n)$ and $(e^{tA})_{t\geq0}$ be a contraction semigroup on $L^2(\mathbb{R}^n)$. Assume that there exist $c_1,c_2,a,b,t_0,m>0$ some positive constants with $a<b$ such that the following spectral inequality 
\[
	\forall g\in L^2(\mathbb{R}^n), \forall k\geq1,\quad \Vert\pi_k g\Vert_{L^2(\mathbb{R}^n)}\le e^{c_1k^a}\Vert\pi_kg\Vert_{L^2(\omega)},
\]
and the following dissipation estimate
\[
	\forall g\in L^2(\mathbb{R}^n), \forall k\geq1, \forall 0<t<t_0,\quad \Vert(1-\pi_k)(e^{tA}g)\Vert_{L^2(\mathbb{R}^n)}\le \frac{1}{c_2}e^{-c_2t^mk^b}\Vert g\Vert_{L^2(\mathbb{R}^n)},
\]
hold. Then, there exists a positive constant $C>1$ such that the following observability estimate holds
\[
    \forall T>0, \forall g\in L^2(\mathbb{R}^n),\quad \Vert e^{TA}g\Vert^2_{L^2(\mathbb{R}^n)}\le C\exp\left(\frac{C}{T^{\frac{am}{b-a}}}\right)\int_0^T\Vert e^{tA}g\Vert^2_{L^2(\omega)}\,\mathrm dt.
\]
\end{theorem}

In the rest of this section, in order to alleviate the text, we will denote the spectral projectors associated with the operator $(\Reelle\rho)(\vert D_x\vert)$ as follows
\begin{equation}\label{eq-projspe}
    \pi_{\lambda,\rho} = \mathds 1_{(-\infty,\lambda]}((\Reelle\rho)(\vert D_x\vert)),\quad \lambda\geq0.
\end{equation}
Let us now prove the following spectral estimates.

\begin{prop}\label{prop-spec}
Let $c_0\in(0,1)$, $c_1>0$, $\alpha \in (0,1)$ and let $\gamma_\rho$ be defined by \cref{eq-gammarho}. Let $\omega\subset\mathbb R^n$ be $\gamma_{\rho}$-thick.
Then, there exists a positive constant $c>0$ such that
\[
    \forall\lambda>0, \forall g\in L^2(\mathbb R^n),\quad\Vert\pi_{\lambda,\rho}g\Vert_{L^2(\mathbb R^n)}
    \le c\eu^{c\lambda^{\alpha}}\Vert\pi_{\lambda,\rho}g\Vert_{L^2(\omega)}.
\]
\end{prop}

\begin{remark}\label{rk-kov} 
The proof of \cref{prop-spec} will be based on Kovrikine's spectral estimate \cite[Theorem 3]{Kovrijkine} stating that there exists a universal positive constant $K>0$ depending only on the dimension $n$ such that for all $(\gamma,L)$-thick set $\omega\subset\mathbb R^n$, with $\gamma\in(0,1]$ and $L>0$, for all $\lambda\geq0$ and $g\in L^2(\mathbb R^n)$ such that $\Supp\widehat g\subset B(0,\lambda)$, we have
\begin{equation}\label{eq-kov}
    \Vert g\Vert_{L^2(\mathbb R^n)}
    \le\Big(\frac K{\gamma}\Big)^{K(1+L\lambda)}\Vert g\Vert_{L^2(\omega)}.
\end{equation}
\end{remark}

\begin{proof}[Proof of \cref{prop-spec}] Let us consider some $\lambda>0$ and $g\in L^2(\mathbb R^n)$ be fixed. Recall that by definition of thickness with respect to $\gamma_\rho$, the set $\omega$ is $\gamma_{\rho}(r)$-thick at every scale $r\in(0,r_0]$, meaning that the following estimate holds for every $r\in(0,r_0]$ and $x\in\mathbb R^n$,
\[
    \Leb(\omega\cap B(x,r))\geq\gamma(r)\Leb(B(x,r)).
\]
On the other hand, by definition \eqref{eq-projspe} of the spectral projector $\pi_{\lambda,\rho}$, the function $\widehat{\pi_{\lambda,\rho}}g$ is supported in $B(0,\rho^{\dagger}(\lambda))$, where
\[
    \rho^{\dagger}(\lambda) \coloneqq \sup\big\{\mu\geq0 : \Reelle\rho(\mu)\le\lambda\big\}.
\]
Notice that $\rho^{\dagger}(\lambda)$ is well-defined since $\lim_{+\infty}\Reelle\rho = +\infty$ by assumption. We therefore deduce from the spectral estimate \eqref{eq-kov} that for all $r\in(0,r_0]$,
\[
    \Vert\pi_{\lambda,\rho}g\Vert_{L^2(\mathbb R^n)}\le\Big(\frac K{\gamma_{\rho}(r)}\Big)^{K(1+r\rho^{\dagger}(\lambda))}\Vert\pi_{\lambda,\rho}g\Vert_{L^2(\omega)},
\]
where the constant $K>0$ only depends on the dimension $n$. Assume for now that $\lambda\geq\lambda_\rho$, where $\lambda_\rho>0$ is defined so that
\[
    \forall\lambda\geq\lambda_\rho,\quad \rho^{\dagger}(\lambda)\geq1/r_0.
\]
Then, by choosing $r = 1/\rho^{\dagger}(\lambda)\in(0,r_0]$ in the above estimate, we obtain that
\[
    \Vert\pi_{\lambda,\rho}g\Vert_{L^2(\mathbb R^n)}\le\Big(\frac K{\gamma_{\rho}(1/\rho^{\dagger}(\lambda))}\Big)^{2K}\Vert\pi_{\lambda,\rho}g\Vert_{L^2(\omega)}.
\]
Moreover, since the function $\Reelle\rho$ is continuous, it follows from the definition of $\gamma_{\rho}$ and the definition of $\rho^{\dagger}(\lambda)$ that
\[
    \gamma_{\rho}(1/\rho^{\dagger}(\lambda)) = c_0\exp(-c_1(\Reelle\rho)(\rho^{\dagger}(\lambda))^{\alpha})
    = c_0\exp(-c_1\lambda^{\alpha}).
\]
As consequence, we obtain the following estimate
\[
    \Vert\pi_{\lambda,\rho}g\Vert_{L^2(\mathbb R^n)}\le\Big(\frac K{c_0}\Big)^{2K}e^{2c_1K\lambda^{\alpha}}\Vert\pi_{\lambda,\rho}g\Vert_{L^2(\omega)}.
\]
For the case $0<\lambda<\lambda_{\rho}$, we use again Kovrijkine's estimate to find a $C_1$ such that for every $g\in L^2(\mathbb R^n)$, $\Vert\pi_{\lambda,\rho}g\Vert_{L^2(\mathbb R^n)}\le C_1\Vert\pi_{\lambda,\rho}g\Vert_{L^2(\omega)}$, since $\pi_{\lambda,\rho}(L^2(\mathbb R^n))\subset\pi_{\lambda_{\rho},\rho}(L^2(\mathbb R^n))$. Therefore, for any $c_0>0$, and in particular for $c_0 = 2c_1K$, we have
\[
    \|\pi_{\lambda,\rho}g\|_{L^2(\R^n)}  \leq C_1\eu^{c_0\lambda^\alpha} \|\pi_{\lambda,\rho}g\|_{L^2(\omega)}.
\]
This ends the proof of \cref{prop-spec}.
\end{proof}

\begin{proof}[Proof of \cref{thm-obsgen}] Given a positive time $T>0$ and a set $\omega\subset\mathbb R^n$ being thick with respect to the function $\gamma_{\rho}$ defined in \eqref{eq-gammarho}, we are now in position to prove an observability estimate for the semigroup $(\eu^{-t\overline\rho(\vert D_x\vert)})_{t\geq0}$ from $\omega$ in time $T$. First notice from Plancherel's theorem that the following dissipation estimates hold for any $t>0$, $k\geq1$ and $g\in L^2(\mathbb R^n)$,
\[
    \Vert(1-\pi_{k,\rho})(\eu^{-t\overline\rho(\vert D_x\vert)}g)\Vert_{L^2(\mathbb R^n)}
    = \Vert(1-\pi_{k,\rho})(\eu^{-t(\Reelle\rho)(\vert D_x\vert)}g)\Vert_{L^2(\mathbb R^n)}
    \le e^{-tk}\Vert g\Vert_{L^2(\mathbb R^n)}.
\]
On the other hand, it follows from \cref{prop-spec} that there exists a positive constant $c>0$ such that
\[
    \forall k\geq1, \forall g\in L^2(\mathbb R^n),\quad\Vert\pi_{k,\rho}g\Vert_{L^2(\mathbb R^n)}
    \le c\eu^{c\lambda^{\alpha}}\Vert\pi_{k,\rho}g\Vert_{L^2(\omega)},
\]
where $\alpha\in(0,1)$ is the parameter appearing in the definition of the function $\gamma_{\rho}$. \Cref{thm-obsgen} is then a consequence of \cref{thm-obs}.
\end{proof}

\section{Necessary condition}\label{sec-necessary}
The aim of this section is to prove \cref{thm-mainIII}. To that end, we will use some asymptotics on the evolution of coherent states~\cite[Section~4]{Ko}, that we recall here.

Let $K>0$, and $\Cc \subset \set C$ be as in the statement of \cref{thm-mainIII}. For $\xi=(\xi_i)_{1\leq i \leq n} \in \set C^n$, we will denote $|\xi| = \left(\sum_i |\xi_i|^2\right)^{1/2}$ (the usual norm) and $N(\xi) = \left(\sum_i \xi_i^2\right)^{1/2}$ with principal value of the square root. Notice that for $\xi \in \set R^n$, $N(\xi) = N(\overline \xi) = |\xi|$.

In this section, we choose some quantities as follows:
\begin{enumerate}
 \item let $\lambda >0$ large enough (for instance $\lambda = 4(K+1)$) and $\xi_0 = (\lambda,0,\dotsc,0) \in \set R^n$,
 \item let $\delta>0$ small enough such that for every $\xi\in\set R^n$ and $x\in \set R^n$, $|\xi-\xi_0|<\delta$ and $|x|<\delta$ implies $N(\xi+\xi_0 + ix) \in \Cc$,
 \item let $\chi\in C_c^\infty(B(0,\delta))$ such that $0\leq \chi\leq 1$ and $\chi\equiv1$ on a neighborhood of $0$, say $B(0,\delta_2)$,
 \item finally, for every $0<h\leq h_0$, we define
\begin{equation}\label{eq-def-eps}
    \epsilon(h) \coloneqq T\sup_{|\xi|<\delta, |x|<\delta } h 
    \left|\rho\left(\frac{N(\xi+\xi_0+ix)}h\right)\right|.
\end{equation}

\end{enumerate}
Under these assumptions, we will state some upper and lower bounds on the following function:
\begin{equation}\label{eq-def-gh}
 g_h(t,x) \coloneqq \int_{\set R^n} \chi(\xi-\xi_0) \eu^{-(\xi-\xi_0)^2\!/2h +\iu x\xi/h -t\overline \rho(|\xi|/h)} \diff \xi.
\end{equation}

\begin{prop}\label{th:ighv_asy}
 Under the above assumptions, we have uniformly in $0\leq t \leq T$ and $|x|$ small enough
\begin{align*}
 g_h(t,x) &= (2\pi h)^{n/2} \eu^{ix\xi_0/h - x^2\!/2h -t\overline \rho\big(\frac{N(\xi_0-ix)}h\big) 
+ \bigO\big(\frac{\scriptstyle\epsilon(h)^2}h\big)} 
\big(1 + \bigO(h+\epsilon(h))\big),
\end{align*}
in the limit $h\to 0^+$.
\end{prop}
\begin{proof}
With $\rho_{t,h}(\xi) \coloneqq -t\overline\rho(N(\overline \xi))$ and with the notations of \cite[\S3.2]{Ko}, $g_h(t,x) = I_{t,h,1}(x)$. Then apply \cite[Proposition~3.5]{Ko} adapted in dimension $n$ \cite[\S4.3]{Ko}.
\end{proof}

\begin{prop}\label{th:ighv_upper}
 Let $\eta>0$ and $N \in \set N$. Under the above assumptions, we have uniformly in $0\leq t \leq T$ and $|x|>\eta$
\[
 |g_h(t,x)| \leq \frac C{|x|^N} 
\eu^{-c/h}.
\]
\end{prop}
\begin{proof}
Apply \cite[Proposition~3.7]{Ko} adapted in dimension $n$ \cite[\S4.3]{Ko} with $\rho_{t,h}(\xi) = -t\overline\rho(N(\overline \xi))$. Note that with the notations of this theorem, $\epsilon_t(h)/h = \smallO(c/h)$.
\end{proof}

With these estimates, we can prove \cref{thm-mainIII}.
\begin{proof}[Proof of \cref{thm-mainIII}]
Let $\epsilon>0$ as in the statement of \cref{thm-mainIII}, and let $\epsilon'>0$ small enough (depending on $\epsilon$) to be chosen later.

\step{Observability inequality} As in the proof of \cref{thm-main2} (and see \cite[Theorem~2.44]{coron_book}), the exact null-controllability of the system $(\partial_t+ \rho(|D_x|))f = \mathds1_\omega u$ in time $T$ is equivalent to the following observability inequality: for every $g_0 \in L^2(\R)$, the solution $g$ of 
\begin{equation}\label{eq-gen-frac-adj}
    \partial_tg(t,x) + \overline\rho(|D_x|)g(t,x) = 0,\quad g(0,\cdot) = g_0,
\end{equation} 
satisfies
\begin{equation}\label{eq-obs}
    \|g(T,\cdot)\|_{L^2(\R^n)} \leq C \|g\|_{L^2((0,T)\times \omega)}.
\end{equation}

Throughout this proof, $c$ and $C$ denote constants that can change from line to line.

\step{Choice of test functions} We want to find a lower bound on $\Leb(\omega \cap B(x,L))$ by testing the observability inequality on $g_h$ defined in \cref{eq-def-gh}. Since the equation~\eqref{eq-gen} is invariant by translation, we may assume that $x = 0$. Notice that $g_h$ satisfies
\begin{equation*}
    g_h(t,x) =  h^n \int_{\R^n} \chi(h\xi-\xi_0) \eu^{-(h\xi - \xi_0)^2\!/2h + \iu x \xi - t \overline \rho(|\xi|)}\diff \xi.
\end{equation*}
Thus, $g_h$ is a solution to the heat-like equation~\eqref{eq-gen-frac-adj}.

\step{Lower bound on $g_h$}
Let $R \in(0,1)$ be such that for every $A > 2K$, $\overline{B(A,A R)}\subset \Cc$. According to Harnack's inequality~\cite[Chapter X, Theorem 2.14]{Conway}, if $A>2K$ and $|\mu|<AR$, we have
\[
\Re\rho\left(A+\mu\right) \leq \frac{AR+|\mu|}{AR-|\mu|} \Re\rho\left(A\right).
\]
Hence, if $\delta'<R$, for every $|z|<\delta'$ and $h>0$ small enough, it follows that
\[
\Re\rho\left(\frac{\lambda+z}h\right) \leq \frac{R+\delta'}{R-\delta'} \Re\rho\left(\frac{\lambda}h\right).
\]
We choose $\delta'$ such that $(R+\delta')/(R-\delta') < 1+\epsilon'$.
Reducing $\delta$ if necessary, we may assume that $\delta<\eta$ and that for $\xi,x\in \set R^n$ with $|\xi|<\delta$ and $|x|<\delta$, $|N(\xi_0+\xi+\iu x) - N(\xi_0)|<\delta'$. In this case, we have
\[ 
\Re\rho\left(\frac{N(\xi_0+\xi+\iu x)}h\right) \leq (1+\epsilon') \Re\rho\left(\frac{N(\xi_0)}h\right)
= (1+\epsilon') \Re\rho\left(\frac{\lambda}h\right).
\]
Plugging this into the asymptotic of \cref{th:ighv_asy}, we deduce that for $h$ small enough
\begin{align}
    \|g_h(T,\cdot)\|_{L^2(\R^n)}^2
    &\geq  \|g_h(T,\cdot)\|_{L^2(|x|<\delta)}^2\notag\\
    &\geq c h^n \int_{|x|<\delta} \eu^{-x^2\!/h - 2T \Re\rho(N(\xi_0+\xi-\iu x)/h)+\bigO(\epsilon(h)^2/h)} \diff x\notag\\
    &\geq c h^n \int_{|x|<\delta} \eu^{-x^2\!/h - 2T(1+\epsilon') \Re\rho(\lambda/h)+\bigO(\epsilon(h)^2/h)} \diff x\notag\\
    & \geq c h^{3n/2} \eu^{- 2T(1+\epsilon') \Re\rho(\lambda/h)+\bigO(\epsilon(h)^2/h)}.\notag
    \intertext{
Since we assumed that $\lvert\Im(\rho)| \leq C\Re(\rho)$, we get that $\epsilon(h) \leq h(1+\epsilon')(1+C) \Re (\rho(\lambda/h))$ (see the definition of $\epsilon$ \cref{eq-def-eps}). We deduce that}
    \|g_h(T,\cdot)\|_{L^2(\R^n)}^2
    &\geq c h^{3n/2}\eu^{- 2T(1+\epsilon') \Re\rho(\lambda/h)\big(1+ \bigO(\epsilon(h))\big)}.\notag
\intertext{Finally, for $h$ small enough, $(1+\epsilon')(1+\bigO(\epsilon(h))) < 1+2\epsilon'$. We get}
    \|g_h(T,\cdot)\|_{L^2(\R^n)}^2
    &\geq c h^{3n/2}\eu^{- 2T(1+2\epsilon') \Re\rho(\lambda/h)}.\label{eq-lhs-lower}
\end{align}

\step{Upper bound on $g_h$} Recall that for $h$ small enough, $|\xi|<\delta$ and $|x| < \delta$, $\Re\rho(N(\xi_0+\xi+\iu x)/h) \geq 0$. Hence, the asymptotics stated in~\cref{th:ighv_asy} imply the upper bound
\[
    |g_h(t,x)| \leq C h^{n/2} \eu^{- x^2\!/2h + \bigO(\epsilon(h)/h)}.
\]
Moreover, according to \cref{th:ighv_upper}, for every $h>0$ small enough and for every $0<r<\delta$,
\begin{align}
    \|g_h\|_{L^2((0,T)\times\omega)}^2
    &\leq \|g_h\|_{L^2((0,T)\times\{\delta < |x| \})}^2
    + \|g_h\|_{L^2((0,T)\times\{r < |x| < \delta\}))}^2
    +\|g_h\|_{L^2((0,T)\times\{ |x| < r\}\cap \omega)}^2
    \notag\\
    &\leq C\eu^{-c/h} + Ch^n\eu^{-r^2\!/h} + Ch^n\Leb(B(0,r)\cap \omega).\notag
    \intertext{Reducing $\delta$ if necessary, we may assume that $r^2 < c$, and we can drop the first term of the right-hand side:}
    \|g_h\|_{L^2((0,T)\times\omega)}^2
    &\leq Ch^n\eu^{-r^2\!/h} + Ch^n \Leb(B(0,r)\cap \omega).\label{eq-rhs-upper}
\end{align}

\step{Conclusion}
If the observability inequality~\eqref{eq-obs} holds, according to \cref{eq-lhs-lower,eq-rhs-upper}, there exist $c,C>0$ such that for any $r>0$ and $h>0$ small enough:
\begin{align}
    ch^{n/2}\eu^{- 2T(1+2\epsilon') \Re\rho(\lambda/h)} 
    &\leq C\eu^{-r^2\!/h} + C \Leb(B(0,r)\cap \omega).
    \notag
    \intertext{Since $\rho(\xi) = o(\xi)$ and $\sqrt{h_r 2(T+\epsilon)\Re\rho(\xi_0/h_r)} \leq r$, we get that $h_r \to 0$ as $r\to 0$. Hence, for $r$ small enough, we can apply the previous inequality with $h = h_r$, which gives us }
    ch_r^{n/2}\eu^{- 2T(1+2\epsilon') \Re\rho(\lambda/h_r)} 
    &\leq C\eu^{-r^2\!/h_r} + C\Leb(B(0,r)\cap \omega)\notag\\
    &\leq C\eu^{-2(T+\epsilon)\Re\rho(\lambda/h_r)} + C\Leb(B(0,r)\cap \omega)\notag.
\end{align}
This gives us
\begin{align}
\Leb(B(0,r)\cap \omega) 
&\geq ch_r^{n/2}\eu^{-2T(1+2\epsilon') \Re(\rho(\lambda/h_r))}-c\eu^{-2(T+\epsilon)\Re(\rho(\lambda/h_r))}\notag
    \intertext{Now, since $\ln(\xi) = \smallO(\Re\rho(\xi))$, we get $h_r^{n/2} \geq C\eu^{-2T\epsilon' \Re\rho(\lambda/h_r)}$. For small enough $h_r$ (equivalently, small enough $r>0$), this gives us}
\Leb(B(0,r)\cap \omega) 
&\geq c\eu^{-2T(1+3\epsilon') \Re(\rho(\lambda/h_r))}-c\eu^{-2(T+\epsilon)\Re\rho(\lambda/h_r)}.\notag
\intertext{If we choose $\epsilon'$ small enough, the second term of the right-hand side is negligible, and we get for $h$ small enough}
\Leb(B(0,r)\cap \omega) 
&\geq c\eu^{-2(T+\epsilon) \Re\rho(\lambda/h_r)}.\notag
\end{align}
Dividing this inequality by $\Leb(B(x,r))$ gives the claimed inequality.
\end{proof}

\section{Examples of $\gamma$-thick sets}\label{sec-examples}
In this section, we construct non-trivial sets that are $\gamma$-thick for any given function $\gamma \colon (0,r_0]\to [0,1]$. This construction is based on Smith-Volterra-Cantor sets.

\subsection{Thickness of Smith-Volterra-Cantor sets} Let us first recall the definition of Smith-Volterra-Cantor sets.
\begin{definition}[Smith-Volterra-Cantor sets]
Let $(\tau_n)_{n\in \N}$ be a sequence of real numbers such that $0<\tau_n<1$. For $n\in \N$, let $K_n$ be the closed subset of $[0,1]$, finite union of closed disjoint intervals, defined inductively by the following procedure.
\begin{itemize}\label{def-SVC}
    \item $K_0 \coloneqq [0,1]$.
    \item If $K_n = \bigcup_k I_{nk}$, where the $(I_{nk})_k$ are disjoint closed intervals, remove from $I_{nk}$ the middle part of size $\tau_n\Leb(I_{nk})$ and call the resulting sets\footnote{I.e., if $I_{nk} = [a_{nk},b_{nk}]$, set $b'_{nk} \coloneqq (a_{nk}(1+\tau_n) + b_{nk}(1-\tau_n))/2$ and $a'_{nk} = (a_{nk}(1-\tau_n) + b_{nk}(1+\tau_n))/2$, and finally $I'_{nk} \coloneqq [a_{nk},b'_{nk}]\cup [a'_{nk},b_{nk}]$.} $I'_{nk}$. Then set $K_{n+1}\coloneqq \bigcup_k I'_{nk}$.
\end{itemize}
Let $K \coloneqq \bigcap_{n\in\N} K_n$. The set $K$ is the Smith-Volterra-Cantor set associated to the sequence $(\tau_n)_n$.
\end{definition}

At each step in the construction of a Smith-Volterra-Cantor set, we remove a subset of measure $(1-\tau_n)\Leb(K_n)$ from $K_n$ (see \cref{def-SVC}), hence:
\begin{prop}
With the notations of \cref{def-SVC}, $\Leb(K_n) = \prod_{k=0}^{n-1} (1-\tau_n)$ and $\Leb(K) = {\prod_{n=0}^{+\infty} (1-\tau_n)}$. In particular, $\Leb(K) > 0$ if and only if $\sum_n \tau_n < +\infty$.
\end{prop}

Our first result is some upper and lower bounds on the thickness of the complement of Smith-Volterra-Cantor sets.
\begin{prop}\label{th-SVC-thick}
Let $(\tau_n)_n\in (0,1)^\N$. Assume that $\sum_n \tau_n < +\infty$. Let $K$ be the associated Smith-Volterra-Cantor set and $\omega \coloneqq \R \setminus K$. There exist $c>0$, $C>0$ and $r_0>0$ such that for every $0<r<r_0$,
\[
\frac1{24}\ \sum_{\rule{0em}{1.7ex}\mathclap{k\geq \log_2(3\Leb(K)/r)}}\ \tau_k
\leq \inf_{x\in \R}\dfrac{\Leb(\omega \cap B(x,r))}{\Leb(B(x,r))} 
\leq 6\ \sum_{\mathclap{\quad k\geq \log_2(\Leb(K)/4r)}} \ \tau_k,
\]
where $\log_2$ is the base $2$ logarithm $\log_2(x) = \ln(x)/\ln(2)$.
\end{prop}
With a more careful analysis in the proof below, it seems we could improve this inequality by replacing the $\log_2(3\Leb(K)/r)$ by $\log_2(\kappa \Leb(K)/r)$ with some $\kappa <3$. We don't know what the optimal $\kappa$ is. We don't pursue this because we don't need such a sharp estimate. In the same vein, the factors $1/{24}$ and $6$ are not optimal either.

\begin{proof}
Remark that we only need to estimate $\Leb(\omega \cap B(x,r))$ for $x\in [0,1]$. Indeed, if for instance $x>1$, $\omega \cap B(x,r)$ contains $[x,x+r]$, hence $\Leb(\omega\cap B(x,r))/\Leb(B(x,r)) \in [1/2,1]$.

\step{Notations and preliminary computations}
In this proof, we denote by $I_{nk}$ the intervals that appears in the construction of $K$, as defined in \cref{def-SVC}. We denote the length of $I_{nk}$ (which does not depend on $k$) by $\ell_n$. We have
\[
    \ell_n = \frac{1-\tau_{n-1}}2 \ell_{n-1}.
\]
Notice that
\[
\Leb(I_{nk}\cap \omega) = \Leb(I_{nk}) - \Leb(I_{nk}\cap K) = \ell_n\big(1- \prod_{k\geq n} (1-\tau_k)\big).
\]
In addition, we can estimate the right-hand side in the following way
\begin{IEEEeqnarray*}{rCl"s}
    1-\prod_{k\geq n}(1-\tau_k)
    &=& 1-\exp\Big(\sum_{k\geq n} \ln(1-\tau_k)\Big)\\
    &=& 1-\exp\Big(\sum_{k\geq n} -\tau_k(1+o_k(1))\Big)&(because $\tau_k \to0$)\\
    &=& 1-\exp\Big(-(1+o_n(1))\sum_{k\geq n}\tau_k\Big) &(because $\sum \tau_k<+\infty$)\\
    &=& 1-\Big(1-(1+o_n(1))\sum_{k\geq n}\tau_k\Big)&(because $\sum\limits_{k\geq n} \tau_k \xrightarrow[n\to\infty]{}0$)\\
    &=& (1+o_n(1))\sum_{k\geq n}\tau_k.
\end{IEEEeqnarray*}
Finally, multiplying by $\ell_n$,
\begin{equation}\label{eq-mes-Ink-w}
    \Leb(I_{nk}\cap \omega) = (1+o_n(1))\ell_n \sum_{k\geq n} \tau_k .
\end{equation}

\step{Lower bound when $r$ is comparable to $\ell_n$}\label{step-SVC-lb}
Let $r>0$ and $n\in \N$ be such that $2\ell_n \leq r \leq 6 \ell_n$. Let $x \in [0,1]$.

If $\omega \cap B(x,r)$ contains an interval of length $\geq \ell_n /2$, $\Leb(\omega \cap B(x,r)) \geq \ell_n /2.$

If that is not the case, then, $\distance(x,K_n) < \ell_n/4$. Since $r \geq 2\ell_n$, this implies that $B(x,r)$ contains some $I_{nk}$. Hence, according to \cref{eq-mes-Ink-w},
\[
 \Leb(\omega\cap B(x,r)) \geq\Leb(\omega\cap I_{nk}) = (1+o_n(1)) \ell_n \sum_{k\geq n}\tau_k.
\]

Putting the two cases together, and assuming that $n$ is large enough:
\[
\inf_{x\in \R}\frac{\Leb(\omega \cap B(x,r))}{\Leb(B(x,r))} 
\geq \frac{\ell_n}{2r}\min\Big((1+\smallO_n(1)) \sum_{k\geq n} \tau_k,\frac12\Big) \geq \frac{\ell_n}{4r} \sum_{k\geq n} \tau_k.
\]
Since $r\leq  6 \ell_n$,
\begin{equation}\label{eq-SVC-lb}
    \inf_{x\in \R}\frac{\Leb(\omega \cap B(x,r))}{\Leb(B(x,r))} \geq
    \frac1{24} \sum_{k\geq n} \tau_k,
\end{equation}
this inequality being valid whenever $n$ is large enough and when $2\ell_n \leq r \leq 6 \ell_n$.

\step{Upper bound when $r$ is comparable to $\ell_n$}\label{step-SVC-ub}
Let $r>0$ and $n\in \N$ be such that $\ell_n/3 \leq 2r \leq \ell_n$. Let $x \in [0,1]$ in the middle of a $I_{nk}$, so that $B(x,\ell_n/2) = I_{nk}$. Then $B(x,r) \subset I_{nk}$, and according to \cref{eq-mes-Ink-w},
\begin{equation*}
    \Leb(\omega \cap B(x,r)) 
    \leq \Leb(I_{nk}\cap \omega)
    = (1+o_n(1))\ell_n \sum_{k\geq n} \tau_k.
\end{equation*}
Since, $\ell_n/ 3\leq 2r$, and assuming $n$ is large enough
\begin{equation}\label{eq-SVC-ub}
\inf_{x\in \R}\frac{\Leb(\omega \cap B(x,r))}{\Leb(B(x,r))}\leq (1+\smallO_n(1))\frac{\ell_n}{2r} \sum_{k\geq n} \tau_k \leq 6 \sum_{k\geq n} \tau_k,
\end{equation}
this inequality being valid whenever $n$ is large enough and when $\ell_n/3 \leq 2r \leq \ell_n$.

\step{Solving the inequality $a\ell_n \leq r \leq b\ell_n$}\label{step-SVC-ineq} Let $r>0$ and $0<\kappa <1$. Set $n(r) \coloneqq \lceil \log_2(\Leb(K)/r)\rceil$, where $\lceil \cdot\rceil$ is the ceiling function. We aim to prove that for $r$ small enough, $\kappa \ell_{n(r)} \leq r \leq 2 \ell_{n(r)}$.

According to the definition of $I_{nk}$, $\ell_n = 2^{-n}\prod_{k=0}^{n-1} (1-\tau_k)$. Recall that $\Leb(K)=\prod_{k=0}^{+\infty}(1-\tau_k) >0$. Define $q_n$ by $\Leb(K)q_n = \prod_{k=0}^{n-1}(1-\tau_k)$, i.e., $q_n = \big(\prod_{k\geq n} (1-\tau_k)\big)^{-1}$. Then $q_n >1$ and $q_n \to 1$ as $n\to +\infty$. With this notation, we have the equivalences
\begin{IEEEeqnarray*}{rCl}
    \kappa\ell_n \leq r \leq 2 \ell_n
    &\eq& 2^{-n} \kappa \Leb(K)q_n \leq r \leq 2^{1-n} \Leb(K)q_n\\
    &\eq& -n +\log_2(\kappa \Leb(K)q_n) \leq \log_2(r) \leq 1 -n + \log_2(\Leb(K)q_n)\\
    &\eq& \log_2\Big(\frac{\Leb(K)}r\Big) + \log_2(\kappa)+ \log_2(q_n) \leq n \leq \log_2\Big(\frac{\Leb(K)}r\Big) + 1 + \log_2(q_n).\IEEEyesnumber \label{eq-L-n}
\end{IEEEeqnarray*}
According to the definition of the ceiling function,
\[
\log_2\Big(\frac{\Leb(K)}r\Big) \leq n(r) < \log_2\Big(\frac{\Leb(K)}r\Big) +1.
\]
Since $\log_2(\kappa) <0$, $\log_2(q_n)>0$ and $q_n\to 1$, the inequalities on the right-hand side of \cref{eq-L-n} are satisfied for $n=n(r)$ and small enough $r>0$. Hence, for $r$ small enough, $\kappa \ell_{n(r)}\leq r \leq 2\ell_{n(r)}$, as claimed.

\step{Conclusion}
Applying the previous step with $\kappa = 2/3$, and replacing $r$ by $r/3$, we see that lower bound~\eqref{eq-SVC-lb} holds for $n = \big\lceil \log_2\big(\frac{3\Leb(K)}{r}\big)\big\rceil$ when $r$ is small enough. Hence, the lower-bound stated in \cref{th-SVC-thick} holds.

Applying step \ref{step-SVC-ineq} with $\kappa = 2/3$ and $r$ replaced by $4r$, we get that the upper bound~\eqref{eq-SVC-ub} holds with $n = \big\lceil \log_2\big(\frac{\Leb(K)}{4r}\big)\big\rceil$ and $r$ small enough, which gives the stated upper bound.
\end{proof}

\subsection{Construction of \texorpdfstring{$\gamma$}{}-thick sets}
Let $\gamma\colon (0,r_0]\to [0,1]$ be such that $\gamma(r)\to 0$ as $r\to 0$.\footnote{The case $\gamma(r) \to 0$ as $r\to 0$ is the interesting one. In fact, we claim that if $\limsup_{r\to 0} \gamma(r) > 0$ and $\omega \subset \R^n$ is thick with respect to $\gamma$, then $\R^n\setminus \omega$ is negligible. Indeed, in this case, the definition of $\gamma$-thickness tells us that for all $x \in \R^n$, $\limsup_{r\to 0} \Leb(\omega \cap B(x,r))/\Leb(B(x,r)) >0$. On the other hand, Lebesgue's differentiation theorem applied to $\mathds 1_\omega$ implies that for almost every $x\notin \omega$, $\lim_{r\to 0} \Leb(\omega \cap B(x,r))/\Leb(B(x,r)) =0$, which is only possible if $\R^n\setminus \omega$ is negligible.} Let us explain how we can use \cref{th-SVC-thick} to construct a set that is thick with respect to the function $\gamma$.

First, we prove that it is sufficient to treat the case where $\gamma$ is increasing. Indeed, set $\gamma_1(r) \coloneqq \sup_{(0,r]} \gamma$, i.e., the smallest non-decreasing function that is larger or equal than $\gamma$. Finally, we set $\gamma_2 \coloneqq \gamma_1 + \phi(r)$ where $\phi$ is a small increasing function such that $\lim_{r\to 0} \phi(r) = 0$. This function $\gamma_2$ is such that $\gamma_2(r) \to 0$ as $r\to 0$, $\gamma_2$ is increasing and $\gamma_2(r)\geq\gamma(r)$. Hence, a set that is thick with respect to $\gamma_2$ will also be thick with respect to $\gamma$. From now on, we assume that $\gamma$ is increasing.

We are looking for a sequence $(\tau_k)_k \in (0,1)^{\set N}$ such that $\sum \tau_k < + \infty$ and such that the left-hand side of the thickness estimate in \cref{th-SVC-thick} is larger than $\gamma(r)$, at least for small enough $r$. Moreover, if $K$ is the Smith-Volterra-Cantor set associated to $(\tau_k)_k$ we will look for such a sequence such that $\Leb(K) = 1/3$. In other words, we want $0<\tau_k<1$ and:
\begin{gather}
    \Leb(K) = \prod_{k=0}^{+\infty}(1-\tau_k) = \frac13\label{eq-cond-measure},\\
    \exists r_0>0,\ \forall 0<r\leq r_0,\ \gamma(r) \leq \frac1{24}\sum_{\rule{0em}{1.7ex}\mathclap{k\geq \log_2(1/r)}} \tau_k.\label{eq-cond-thick}
\end{gather}

The thickness condition~\eqref{eq-cond-thick} imposes $\sum_{k\geq n}\tau_k \geq 24\sup_{[2^{-n-1},2^{-n})}\gamma = 24 \gamma(2^{-n})$ since $\gamma$ is increasing. Motivated by this, we set $\tau_n \coloneqq 24(\gamma(2^{-n}) - \gamma(2^{-n-1}))$ when it is defined, i.e., for $n\geq n_0 \coloneqq \lceil\log_2(1/r_0)\rceil$. Notice that since $\gamma$ is increasing, we do have $\tau_k>0$. For $n\geq n_0$, we have $\sum_{k\geq n} \tau_k = 24(\gamma(2^{-n})-\lim_{0} \gamma) = 24\gamma(2^{-n}) < +\infty$. For $n<n_0$, we choose some arbitrary value for $\tau_n$, for instance $\tau_n = 1/2$. This sequence $(\tau_k)$ satisfies the thickness condition~\eqref{eq-cond-thick} by construction.

Of course, we might not have the measure condition~\eqref{eq-cond-measure}, and for some $k$, we might not even have $\tau_k<1$. But we can tweak the sequence $(\tau_k)_k$ to ensure these properties. Notice that if we change a finite number of $\tau_k$, the thickness condition~\eqref{eq-cond-thick} still holds (with a smaller $r_0$). There are a finite number of $\tau_k$ that are larger or equal than $1$ (if any), and we can set them to, e.g., $1/2$. Next, if $\Leb(K)>1/3$, increase $\tau_0$ to reduce $\Leb(K)$. And if $\Leb(K)<1/3$, choose $N$ so that $\prod_{k=N+1}^{+\infty}(1-\tau_k)>1/3$, and decrease $\tau_0,\dots,\tau_N$ to increase $\Leb(K)$.

Let us end this construction by noticing that it is almost optimal in the following sense: according to the right-hand side of the thickness estimate in \cref{th-SVC-thick}, the set $\omega = \set R\setminus K$ is such that for $r$ small enough,
\begin{align*}
 \inf_{x\in \set R}\frac{\Leb(B(x,r)\cap \omega)}{\Leb(B(x,r))} 
 &\leq 6\ \sum_{\mathclap{\quad k\geq \log_2(1/12r)}} \ \tau_k\\
 &= 6\times 24 \gamma \left( 2^{-\lceil\log_2(1/12r)\rceil}\right)\\
 &\leq 144 \gamma(12r).
\end{align*}

We summarize this construction in the following proposition:
\begin{prop}\label{th-omega-gamma-thick}
 Let $\gamma\colon (0,r_0]\to [0,1]$ be such that $\gamma(r) \to 0$ as $r\to 0$. There exists a set $\omega \subset \set R$ such that $\Leb(\set R\setminus \omega)>0$ and such that for every $r$ small enough,
 \[
 \gamma(r) \leq \inf_{x\in \set R}\frac{\Leb(B(x,r)\cap \omega)}{\Leb(B(x,r))}.
 \]
 Moreover, if $\gamma$ is increasing, we can choose $\omega = \set R\setminus K$, where $K$ is the Smith-Volterra-Cantor set associated to a sequence $(\tau_n)_n$ such that for $n$ large enough, $\tau_n = 24(\gamma(2^{-n}) - \gamma(2^{-n-1}))$, in which case, for small enough $r$,
 \[
 \inf_{x\in \set R}\frac{\Leb(B(x,r)\cap \omega)}{\Leb(B(x,r))} \leq 144 \gamma(12r).
 \]
\end{prop}

This proposition constructs $\gamma$-thick sets only in dimension $1$. In higher dimension, we can prove that if $\omega_1\subset \set R$ is $\gamma$-thick, then $\omega\coloneqq \omega_1\times \set R^{n-1}$ is thick with respect to  $a_n\gamma(b_n\cdot)$, for some universal constants $a_n, b_n$ that depends only on $n$. Indeed, let $b_n>0$ such that $[-b_n,b_n]^n \subset B(0,1)$. Then,
\begin{IEEEeqnarray*}{rCl"s}
    \Leb(B(x,r)\cap \omega)
    &\geq& \Leb([x-b_nr,x+b_nr]^n \cap \omega) & ($[x-b_nr,x+b_nr]^n\subset B(x,r)$)\\ 
    &\geq& \Leb([x-b_nr,x+b_nr] \cap \omega_1)(2b_nr)^{n-1}& ($(\prod_i A_i)\cap (\prod_i B_i) = \prod_i(A_i\cap B_i)$)\\
    &\geq& \gamma(b_nr)(2b_nr)^n & ($\omega_1$ is thick with respect to $\gamma$).
\end{IEEEeqnarray*}
Thus,
\[
\frac{\Leb(B(x,r)\cap \omega)}{\Leb(B(x,r)} \geq \frac{(2b_n)^{n}}{\Leb(B(0,1))} \gamma(b_nr).
\]

\appendix

\section{Heat-like equation on the torus}
The previous results are stated for equations posed on the whole space. But we can adapt these results when the equation is posed on the torus $\T^n \coloneqq (\R\slash2\pi\Z)^n$. The equation is written in the same way:
\begin{equation}\label{eq-genT}\tag{$E^\T_\rho$}
\begin{cases}
	\partial_tf(t,x) + \rho(\vert D_x\vert)f(t,x) = \mathds 1_\omega u(t,x), \quad (t,x) \in \R^*_+ \times \T^n,\\
	f(0,\cdot) = f_0\in L^2(\T^n),
\end{cases}
\end{equation}
where $\rho(\vert D_x\vert)$ is again defined with the functional calculus, that is to say, for $f\in L^2(\T^n)$ and $k\in \Z^n$, denoting the $k$-th Fourier coefficient of $f$ by $c_k(f)$, we define $\rho(|D_x|)f$ by $c_k(\rho(|D_x|)f) = \rho(|k|)c_k(f)$.

The strongly dissipative case (e.g.~$\rho(\xi) = \xi^s$ with $s>1$) for these equations on bounded domains or compact manifold is well known, by combining Burq and Moyano's estimate~\cite[Theorem 1]{BM} and Lebeau and Robbiano's method, as stated by Duyckaerts and Miller~\cite[Theorem 6.1]{DM} (see also references in those two articles).

For the weakly dissipative case, we have the following straightforward adaptations of \cref{def-a-thick,thm-main2,thm-mainIII}.

\begin{definition}
Given some $r_0>0$ and a function $\gamma\colon(0,r_0]\rightarrow[0,1]$, a set $\omega\subset\T^n$ is said to be \textit{thick relatively to} $\gamma$ (or $\gamma$-thick) when it is measurable and satisfies that for every $r\in(0,r_0]$ and $x\in \T^n$,
\[\Leb(\omega\cap B(x,r))\geq\gamma(r)\Leb(B(x,r)).\]
\end{definition}
\begin{theorem}[\cref{thm-main2} in the torus]\label{thm-main2T} Let $\rho\colon[0,+\infty)\rightarrow\mathbb C$ and $\gamma_\rho$ be as in \cref{thm-main2}, and $\omega \subset \T^n$ be $\gamma_\rho$-thick. For every $T>0$, the parabolic equation \eqref{eq-genT} is null-controllable from $\omega$ in time $T$.
\end{theorem}
\begin{proof}[Sketch of the proof]
Egidi and Veseli\'c proved a version of Kovrijkine's estimate for functions defined on the torus~\cite[Theorem 2.1]{EV2}. The proof of \cref{thm-main2T} is a copy-paste of the one of \cref{thm-main2}, where we replace Kovrijkine's estimate~\eqref{eq-kov} by the aforementionned version on the torus. 
\end{proof}
\begin{theorem}[\cref{thm-mainIII} on the torus]\label{thm-mainIIIT}
Let $K>0$, $\Cc\subset \C$, $\rho\colon \Cc\cup \R_+ \to \C$ and $h_r$ be as in \cref{thm-mainIII}. Let $T>0$ and $\omega\subset \T^d$ be measurable. Assume that the parabolic equation~\eqref{eq-genT} is null controllable from $\omega$ in time $T>0$. 

There exists $\lambda>0$ and $c>0$ such that if $\epsilon>0$, and if $r\mapsto h_r$ is defined for small enough $r$ and such that for every $r$ small enough,
\[
    \sqrt{h_r (2T+\epsilon) \Re\rho\left(\frac{\lambda}{h_r}\right)} \leq r,
\]
then, for every small enough $r> 0$ and $x\in \T^n$, we have
\[
    \frac{\Leb(\omega\cap B(x,r))}{\Leb(B(x,r))} 
    \geq cr^{-n}\exp\left(-2(T+\epsilon) \Re\rho\left(\frac{\lambda}{h_r}\right)\right).
\]
\end{theorem}
\begin{proof}[Sketch of the proof]
Consider $g_h$ as defined by \cref{eq-def-gh} and
\[
\per{g_h}(t,x) \coloneqq \sum_{k\in \Z^n} g_h(t,x+2\pi k).
\]
We can check that $\per{g_h}$ is a solution of the adjoint equation $\partial_tg(t,x) + \overline{\rho}(\vert D_x\vert)g(t,x) = 0$ on the torus (see \cite[\S4.2]{Ko}). Moreover, the terms for $k\neq 0$ are exponentially small. Hence, when estimating the left-hand side $\|\per{g_h}(T,\cdot)\|_{L^2}$ and the right-hand side $\|\per{g_h}\|_{L^2([0,T]\times \omega)}$ of the observability inequality, only the term for $k = 0$ matters. Thus, we can do all the computations of the proof of \cref{thm-mainIII} with $\per{g_h}$ instead of $g_h$.
\end{proof}

\subsection*{Acknowledgements} The second author thanks Pierre Lissy for interesting discussions on this topic.

\end{document}